\documentclass[11pt,a4paper]{amsart}
\usepackage{amssymb,amsmath,amsthm, amsaddr}
\usepackage{commath}
\usepackage{mathtools}

\usepackage{accents}

\usepackage[utf8]{inputenc}

\usepackage[colorlinks]{hyperref}
\hypersetup{
    allcolors = black
}

\newtheorem{theorem}{Theorem}

\renewcommand{\geq}{\geqslant}
\renewcommand{\leq}{\leqslant}

\newcommand{\R}{\mathbb{R}}
\newcommand{\eps}{\varepsilon}
\newcommand{\la}{\langle}
\newcommand{\ra}{\rangle}

\newcommand{\PP}{\mathbb{P}}

\renewcommand{\a}{\mathbf{a}}

\newcommand{\vol}{\mathrm{Vol}}
\newcommand{\0}{\mathbf{0}}

\newcommand{\conv}{\mathrm{conv} \, }
\newcommand{\x}{\mathbf{x}}

\newcommand{\dd}{\, \mathrm{d}}

\newcommand{\erf}{\mathrm{erf}}

\makeatletter
\@namedef{subjclassname@2020}{%
  \textup{2020} Mathematics Subject Classification}
\makeatother

\numberwithin{equation}{section}

\title{Critical central sections of the cube}

\author{Gergely Ambrus}

\address{}

\begin{document}

\thanks{Research of the author was supported by NKFIH grant KKP-133819 and by the EFOP-3.6.1-16-2016-00008 project, which in turn has been supported by the European Union, co-financed by the European Social Fund.
 }



\begin{abstract}
We study the volume of central hyperplane sections of the cube. Using Fourier analytic and variational methods, we retrieve a geometric condition characterizing critical sections which, by entirely different methods, was recently proven by Ivanov and Tsiutsiurupa. Using this characterization result, we prove that critical central hyperplane sections in the 3-dimensional case are all diagonal to a (possibly lower dimensional) face of the cube, while in the 4-dimensional case, they are either diagonal to a face, or, up to permuting the coordinates and sign changes, perpendicular to the vector $(1,1,2,2)$. This shows the existence of non-diagonal critical central sections.
\end{abstract}

\maketitle

\section{History and results}
Let $Q_n = [ -1, 1 ]^n$ denote the standard $n$-dimensional cube, which is  the unit ball of the $\ell_\infty$-norm on $\R^n$. It is a classical question to study the $(n-1)$-dimensional volume of sections of $Q_n$ with hyperplanes containing $\0$. Determining which central sections are of minimal and maximal volume had been at the center of attention for over a century, as this question was already rooted in the works of Laplace~\cite{La1812} and Pólya~\cite{Po13}. Yet, it was not before the 1970's that Hadwiger~\cite{Ha71} proved that {\em minimal} hyperplane sections are parallel to facets of $Q_n$ and thus they are of volume $2^{n-1}$. A few years later, Hensley~\cite{He79} independently re-proved this result using probabilistic methods and also gave an upper bound on the volume of central hyperplane sections. In his celebrated work, Ball~\cite{Ba86} proved that {\em maximal} hyperplane sections are orthogonal to a main diagonal of a 2-dimensional face of $Q_n$, hence, their volume is $\sqrt{2} \cdot 2^{n-1}$. Extensions of these estimates to {\em lower dimensional sections} were proven by Vaaler~\cite{Va79},  Ball~\cite{Ba89}, and Ivanov and Tsiutsiurupa \cite{IT21}, while alternate proofs were given by Nazarov and Podkorytov~\cite{NP00}, and Akopyan, Hubard and Karasev~\cite{AHK19}. Non-central sections were estimated by Moody, Stone, Zach and Zvavitch~\cite{MSZZ13} and König~\cite{K21}. Further extensions to {\em unit balls of $\ell_p$-norms} were studied by Meyer and Pajor~\cite{MP88}, Koldobsky~\cite{K98}, Eskenazis~\cite{E19}, and Liu and Tkocz~\cite{LT20}. Generalizations to {\em Gaussian measures} were given by Zvavitch~\cite{Z08}, Barthe, Guédon, Mendelson and Naor~\cite{BGMN05} and Koldobsky and König \cite{KK12}, while the analogous question for perimeters was studied recently in \cite{KK19}. Aliev~\cite{Al21} determined maximal sections with respect to a certain normalization.  The closely related problem of estimating volumes of {\em central slabs} was discussed by Barthe and Koldobsky~\cite{BK03} and by König and Koldobsky~\cite{KK11}.  For related results and further discussions, see~\cite{Ko05} and~\cite{KY08}.

Determining minimal and maximal sections is only the tip of the iceberg when studying the behaviour of $\vol_{n-1}(Q_n \cap \a^\perp)$ as a function of $\a \in S^{n-1}$ (the unit sphere in $\R^n$). One is tempted to believe that local extremizers are perpendicular to a main diagonal of a $k$-dimensional face of $Q_n$ -- we are going to call these {\em $k$-diagonal directions} and the corresponding hyperplane sections as {\em $k$-diagonal sections.} Hence, up to permuting the coordinates and changing signs, $k$-diagonal directions are of the form $(\frac 1{ \sqrt{k}}, \ldots,\frac 1{ \sqrt{k}}, 0, \ldots, 0)$ where the number of non-zero coordinates is $k$. Without specifying $k$, we will also simply refer to {\em diagonal directions} and {\em diagonal sections}.

As a first step in the analysis, it is essential to compare the volumes of diagonal sections. By probabilistic methods, the Central Limit Theorem implies that the volume of $k$-diagonal sections for $k \approx n$ is about $\sqrt{6 / \pi} \cdot 2^{n-1}$, which is slightly less than the volume of 2-diagonal sections, that is $\sqrt{2} \cdot 2^{n-1}$. In their recent work, Bartha, Fodor, and González Merino~ \cite{BFGM21} proved that for each fixed $n \geq 3$ the volume of $k$-diagonal sections form a strictly increasing sequence for $k \geq 3$, which is sandwiched between the values taken at $k=1$ and $k=2$. Considering sections of {\em arbitrary dimension},  by using geometric methods, analyzing local modifications, and studying the relationship with frames, Ivanov and Tsiutsiurupa \cite{IT21} established necessary conditions for sections in order to have locally maximal volume.

In this note we apply Fourier analytic methods to study {\em critical} central hyperplane sections: these are the sections $\a^\perp \cap Q_n$ whose normal vector $\a$ is a critical point on $S^{n-1}$ with respect to the volume of the central section of $Q_n$. Such normal vectors will be referred to as {\em critical directions}. We retrieve the main condition of Ivanov and Tsiutsiurupa \cite{IT21} (see Theorem 1.1, Condition 4 therein) for $(n-1)$-dimensional sections being locally maximal. Yet, the  Fourier analytic approach yields a more transparent proof for the statement. Our argument is reminiscent of the work of Koldobsky and König~\cite{KK11} -- for that connection, see the remark following the proof of Theorem~\ref{thm1}.

Below and later on, $S_i$ denotes the facet of $Q_n$ corresponding to the $i$th coordinate being~1. That is, $S_i = \{ \x = (x_1, \ldots, x_n) \in Q_n : \ x_i = 1\}$. The vectors $\a, \x \in \R^n$ will always have coordinates $\a = (a_1, \ldots, a_n)$ and $\x = (x_1, \ldots, x_n)$.  As mentioned before, $S^{n-1}$ is the unit sphere of $\R^n$.

\begin{theorem}\label{thm1}
The unit vector $\a = (a_1, \ldots, a_n) \in S^{n-1}$ is a critical direction with respect to the central section volume function $\vol_{n-1} (Q_n \cap \a^\perp)$ if and only if it is parallel to a main diagonal of a 2-dimensional face of $Q_n$, or there exists some $\mu >0$ for which
\begin{equation}\label{thm1eq}
  \vol_{n-1}(\conv(\0 \cup ( S_k \cap \a^{\perp})) = \mu (1 - a_k^2)
\end{equation}
holds true for each $k = 1, \ldots, n$.
\end{theorem}

The condition guaranteed by Theorem~\ref{thm1} may be used  to calculate critical directions. With the aid of probabilistic methods, we demonstrate this for $n \leq 4$. Based on computational evidence, it has been widely believed that all critical directions are diagonal. We confirm this conjecture in dimension 3, but disprove it in dimension 4. 

\begin{theorem}\label{thm2}
For $n = 2,3$, critical central hyperplane sections of $Q_n$ are exactly the diagonal sections.
\end{theorem}

\begin{theorem}\label{thm3}
For $n = 4$, critical central hyperplane sections of $Q_4$ are either diagonal, or their normal vector is $\frac 1 {\sqrt{10}}(1,1,2,2)$ up to permuting coordinates and changing signs.
\end{theorem}

\section{A geometric characterization of critical sections}

Theorem~\ref{thm1} will be proved using Fourier analytic tools and variational methods -- see~\cite{Ko05} for detailed theoretical background. Naturally, the following classical formula, dating back to Pólya~\cite{Po13}, lies at the core of the arguments: For any unit vector $\a  \in S^{n-1}$,
\begin{equation}\label{eq_laplace}
\vol_{n-1}(Q_n \cap \a^\perp) = \frac {2^{n-1}}{\pi} \int_{-\infty}^\infty \prod_{i=1}^n \frac {\sin a_i t}{a_i t} \dd t.
\end{equation}
Here and later on, $\frac{\sin(0)}{ 0}$ is understood to be 1.

We will use the following generalization of \eqref{eq_laplace}. For arbitrary non-zero $\a \in \R^n$, introduce the  {\em parallel section function} $s_\a(.)$ defined on $\R$ as
\begin{equation}\label{def_section}
s_{\a}(r) = \vol_{n-1}(\x \in Q_n: \ \la \x, \a \ra = r).
\end{equation}
Note that $s_{\a}(r)$ is the $(n-1)$-dimensional volume of the hyperplane section of $Q_n$ orthogonal to $\a$ {\em at distance $\frac {r}{ |\a|}$ from the origin}. In particular, $s_{\a}(r)$ is {\em not} invariant under scaling of~$\a$.

Furthermore, introduce  the {\em normalized central section function} $\sigma(.)$ defined on $\R^n \setminus \{ \0 \}$ as
\begin{equation}\label{def_sigma}
\sigma(\a) = \frac{\pi}{2^{n-1}}s_\a(0).
\end{equation}

For a continuous random variable $X$, let $f_X(.)$ denote its density function, $F_X(.)$ its distribution function, and $\varphi_X(.)$ its characteristic function.

Let now $X_1, \ldots, X_n$ be independent random variables, each distributed uniformly on $[-1,1]$.  The joint distribution of $X_1, \ldots, X_n$ induces the normalized Lebesgue measure on~$Q_n$. Accordingly, for arbitrary $\0 \neq \a  \in \R^n$ and  $r \in \R$,
\[
\PP\Big( \big|\sum_{i=1}^n a_i X_i - r\big| \leq \eps \Big) = \frac 1 {2^n} \vol_n \Big( \x \in Q_n: \ \big|\la \x, \a \ra  - r \big|\leq \eps \Big).
\]
Thus, letting $\eps \rightarrow 0$, we deduce by \eqref{def_section} that
\begin{equation}\label{fs}
  f_{\sum_{i=1}^n a_i X_i}(r) = \frac {1}{2^n | \a |} s_\a(r)
\end{equation}
as long as $ f_{\sum_{i=1}^n a_i X_i}(r)$ exists, that is, $\a$ and $r$ are not of the form $ \a = c \cdot \mathbf{e}_j$ and $|r| = |c|$, where $\mathbf{e}_j$ denotes the $j$th standard basis vector, and $c \in \R$.

Excluding the above degenerate cases, we can express $f_{\sum_{i=1}^n a_i X_i}(r)$ using the inverse Fourier transform. As is well known, the characteristic function of $\sum_{i=1}^n a_i X_i$ is given by
\begin{equation}\label{charfn}
\varphi_{\sum_{i=1}^n a_i X_i}(t) = \prod_{i=1}^n \frac{\sin a_i t}{a_i t}.
\end{equation}
Hence, by standard Fourier inversion,
\begin{equation}\label{fsumint}
f_{\sum_{i=1}^n a_i X_i}(r)  = \frac 1 {2 \pi} \int_{-\infty}^\infty \prod_{i=1}^n \frac {\sin a_i t}{a_i t} \cdot  \cos rt \dd t.
\end{equation}
Therefore, by \eqref{fs},
\begin{equation}\label{sint}
  s_\a(r) = \frac{2^{n-1}|\a|}{\pi} \int_{-\infty}^\infty \prod_{i=1}^n \frac {\sin a_i t}{a_i t} \cdot  \cos rt \dd t
\end{equation}
and, by \eqref{def_sigma}, we recover \eqref{eq_laplace}:
\begin{equation}\label{sigmaint}
\sigma(\a) = |\a|\int_{-\infty}^\infty \prod_{i=1}^n \frac {\sin a_i t}{a_i t}  \dd t.
\end{equation}

\begin{proof}[Proof of Theorem~\ref{thm1}]
Assume that $\a  \in S^{n-1}$ is a critical direction. Since the constraint on $\a$ is expressed by the equation $a_1^2 + \ldots + a_n^2 = 1$, and $\sigma(\a)$ is differentiable, the method of Lagrange multipliers implies that there exists a constant $\lambda \in \R$ so that for each $k = 1, \ldots, n$,
\begin{equation}\label{lak}
  \lambda a_k  = \frac{\partial}{\partial a_k} \sigma(\a).
\end{equation}
Referring to \eqref{sigmaint}, the Leibniz integral rule shows that for each $k$ with $a_k \neq 0$,
\begin{equation}\label{pd}
  \frac{\partial}{\partial a_k} \sigma(\a) = \int_{-\infty}^\infty \prod_{ i \neq k} \frac {\sin a_i t}{a_i t} \cdot \Big( \frac{\cos a_k t} {a_k} - \frac{\sin a_k t}{a_k^2 t} \Big) \dd t.
\end{equation}
Therefore, multiplying \eqref{lak} by $a_k$ shows that
\begin{equation}\label{lak2}
  \lambda a_k^2 + \sigma(\a) = \int_{-\infty}^\infty \prod_{ i \neq k} \frac {\sin a_i t}{a_i t} \cdot \cos a_k t \dd t
\end{equation}
whenever $a_k \neq 0$. Since $\frac{\partial}{\partial a_k} \sigma(\a) =0$ if $a_k = 0$, the above equality holds in fact for all values of $k=1, \ldots, n$.

Next, we calculate the value of $\lambda$. To that end, let $\varphi(t) := \varphi_{\sum_{i=1}^n a_i X_i}(t)$.  Then by~\eqref{charfn},
\[
\varphi'(t) = \frac 1 t \cdot \sum_{k=1}^n \Big( \prod_{i \neq k} \frac {\sin a_i t}{a_i t} \cdot \cos a_k t \Big) - \frac n t \cdot  \varphi(t).
 \]
Thus, summing \eqref{lak2} over $k = 1, \ldots, n$ leads to
\begin{align*}
\lambda & = \lambda \sum_{k=1}^n a_k^2\\
&=  - n\, \sigma(\a) + \int_{ -\infty} ^ \infty t \varphi'(t)  + n \varphi(t) \dd t\\
&= \int_{ -\infty} ^ \infty t \varphi'(t) \dd t \\
&= \Big[ t \varphi(t) \Big]_{- \infty}^\infty - \int_{ -\infty} ^ \infty \varphi(t) \dd t \\
& = - \sigma(\a)
\end{align*}
whenever $n \geq 2$.

Introduce $\widetilde\a_k = (a_1, \ldots, a_{k-1}, a_{k+1}, \ldots, a_n) \in \R^{n-1}$. If $\widetilde\a_k  = \0$, then $a_k = \pm 1$ and $\a$ is a critical direction corresponding to a minimal central section. Note that \eqref{thm1eq} holds in this case. Otherwise, by \eqref{sint}, equation \eqref{lak2} translates to
\begin{equation}\label{sigmavol}
\sigma(\a) \cdot (1 - a_k^2) = \int_{-\infty}^\infty \prod_{ i \neq k} \frac {\sin a_i t}{a_i t} \cdot \cos a_k t \dd t = \frac{\pi}{2^{n-2}\sqrt{1 - a_k^2}}\, s_{\widetilde\a_k}(a_k)
\end{equation}
unless only one coordinate of $\widetilde\a_k$ is non-zero, and its absolute value equals that of $a_k$. In this case, the critical direction is, up to permutations and sign changes of the coordinates, $\a = (\frac 1 {\sqrt{2}},\frac 1 {\sqrt{2}}, 0, \ldots, 0 )$ -- that is, $a$ is parallel to a main diagonal of a 2-dimensional face of~$Q_n$. These are indeed critical directions, as they correspond to the maximal central sections of~$Q_n$ \cite{Ba86}.

From now on, assume that $\a^\perp \cap Q_n$ is not a minimal or a maximal section, and consider the $(n-2)$-dimensional section
\[
S_k \cap \a^{\perp} = \{ \x \in Q_n: \ \la \x, \a \ra = 0 \textrm{ and }x_k = 1 \}.
\]
Then (see also \cite{KK19})
\begin{equation}\label{volska}
  \vol_{n-2} (S_k \cap \a^{\perp}) = s_{\widetilde\a_k}(-a_k) = s_{\widetilde\a_k}(a_k).
\end{equation}
An elementary geometric computation shows that the distance between $\0$ and $S_k \cap \a^{\perp}$ is $\frac {1} {\sqrt{1 - a_k^2}} = 1/ |\widetilde\a_k|$. Therefore, by \eqref{sigmavol} and \eqref{volska},
\begin{equation*}
\begin{split}
\vol_{n-1}(\conv(\0 \cup ( S_k \cap \a^{\perp})) &= \frac {\vol_{n-2} (S_k \cap \a^{\perp})}{(n - 1)|\widetilde\a_k| } = \frac{s_{\widetilde\a_k}(a_k)}{(n - 1)|\widetilde\a_k|} \\ &= \frac{2^{n-2} \sigma(\a)}{(n-1) \pi} \cdot (1 - a_k^2)
\end{split}
\end{equation*}
which shows that \eqref{thm1eq} holds indeed. 

For the reverse direction, it suffices to assume that there exists $\mu >0$ so that  \eqref{thm1eq} is satisfied. Summing over $k = 1, \ldots, n$ we obtain half the volume of the section $\a^\perp \cap Q_n$ (note that we only take the cones spanned by sides corresponding to $x_i =1$). Hence $
\frac 1 2 s_{\a}(0) = (n-1) \mu$. Therefore,  
\[
\mu = \frac{s_{\a}(0)}{2 (n-1)} = \frac{2^{n-2}}{(n-1)\pi} \sigma(\a)
\]
and \eqref{sigmavol} holds for every $k$. This implies \eqref{lak2} with $\lambda = -\sigma(\a)$, which in turn yields \eqref{lak} whenever $a_k \neq 0$. Since \eqref{lak} also holds trivially for $a_k =0$, we obtain that $\nabla(|\a|)$ and $\nabla(\sigma(\a))$ are indeed parallel to each other, hence $\a$ is a critical direction.
\end{proof}

Note that if $\a$ is not a normal direction of a minimal or a maximal section of $Q_n$,  then \eqref{thm1eq} shows that $s_{\widetilde\a_k}(a_k) \neq 0$ for each $k$, which via \eqref{fsumint} and \eqref{sint} translates to the condition
\begin{equation}\label{absa}
|a_k| < \sum_{i \neq k}|a_i|.
\end{equation}

I would like to point out that the above proof is in the same spirit as the one given by König and Koldobsky \cite{KK11} used for determining central slabs of the cube of extremal volume. This connection is not gratuitous: for each $a_k \neq 0$, applying an orthogonal projection of $Q_n \cap \a^\perp$ onto $S_k$ shows that
\[
s_{\a}(0) = \frac 1 {a_k} \vol_{n-1}(\x \in Q_{n-1}: |\la \x, \a_k \ra| \leq a_k).
\]
Yet, the two problems behave quite differently: for the question regarding the volume of  slabs, the 3-dimensional case is already surprisingly complex with a large number of critical sections, while for the present problem, the behaviour of $\sigma(\a)$ is still fairly simple when $n = 4$, as shown by Theorem~\ref{thm3}.

\section{Critical sections for $n=2,3$}

Our goal in this section is to prove that critical central hyperplane sections are all diagonal when $n \leq 3$. That amounts to showing that given a normal vector $\a$ which is critical with respect to $\sigma(\a)$, all of its non-zero coordinates are equal up to sign changes.

We start the proof of Theorem~\ref{thm2} by noting that for $n=2$ the statement follows by an elementary geometric observation: maximal sections are $2$-diagonal, minimal sections are $1$-diagonal, and the length of the central sections changes monotonously between these extrema. Therefore, we restrict our study to the case $n=3$.

Let $\a \in S^{n-1}$ be a critical normal direction. If $a_k = 0$ holds for some $k$, then $\widetilde{\a}_k$ needs to be a critical direction for  $S_k$ as well. Therefore, we may assume that all coordinates of $\a$ are non-zero. Furthermore, by symmetry, we may assume that $a_k > 0$ for each $k \leq n$. Then, our goal is to show that all the coordinates $a_k$ are identical.

We will compare two coordinates, say, $a_1$ and $a_2$. Formulae \eqref{fsumint} and \eqref{sigmavol} yield
\[
\frac{1}{1 - a_1^2}\, f_{\sum_{i\neq 1} a_i X_i} (a_1) = \frac{1}{1 - a_2^2} \, f_{\sum_{i \neq 2} a_i X_i} (a_2).
\]
Since
\[
 f_{\sum_{i\neq 1} a_i X_i} (a_1) = \frac{1}{2 a_2} \, \int_{a_1 - a_2}^{a_1 + a_2} f_{\sum_{i=3}^n a_i X_i}(x) \dd x
\]
and
\[
 f_{\sum_{i\neq 2} a_i X_i} (a_2) = \frac{1}{2 a_1} \, \int_{a_2 - a_1}^{a_2 + a_1} f_{\sum_{i=3}^n a_i X_i}(x) \dd x,
\]
this is equivalent to
\begin{equation}\label{a1a2}
\frac{1 - a_2^2}{a_2} \int_{a_1 - a_2}^{a_1 + a_2} f_{\sum_{i=3}^n a_i X_i}(x) \dd x = \frac{1 - a_1^2}{a_1}  \int_{a_2 - a_1}^{a_2 + a_1} f_{\sum_{i=3}^n a_i X_i}(x) \dd x.
\end{equation}
The above equation must hold true for any pair of coordinates in place of $a_1$ and $a_2$ as well.

\begin{proof}[Proof of Theorem~\ref{thm2} for $n = 3$.]
We may assume that $ 0 < a_1 \leq a_2 \leq a_3 <1$. By \eqref{absa}, $a_3 < a_2 + a_1$. Therefore, the sum of any two of the coordinates is larger than the third one, while their difference is smaller than that. Thus, \eqref{a1a2}  reads as
\[
\frac{1 - a_2^2}{a_2}\Big(\frac 1 2 + \frac{a_2 - a_1}{2 a_3}\Big) = \frac{1 - a_1^2}{a_1}\Big(\frac 1 2 - \frac{a_2 - a_1}{2 a_3}\Big)
\]
which implies that
\begin{equation}\label{a1a2_2}
a_1 + a_2 - a_3 - a_1 a_2^2 - a_1^2 a_2 - a_1 a_2 a_3 = 0
\end{equation}
if $a_1 \neq a_2$.

First, assume that all three coordinates of $\a$ are different. By swapping the role of $a_1, a_2$ and $a_3$ above,  \eqref{a1a2_2} modifies to
\begin{align*}
a_1 - a_2 + a_3 - a_1 a_3^2 - a_1^2 a_3 - a_1 a_2 a_3  &= 0\\
-a_1 + a_2 + a_3 - a_2 a_3^2 - a_2^2 a_3 - a_1 a_2 a_3 &= 0.
\end{align*}
Summing the above equations along with  \eqref{a1a2_2} results in
\[
(a_1 + a_2 + a_3) (1 - a_1 a_2 - a_1 a_3 - a_2 a_3)= 0,
\]
hence,
\[
a_1 a_2 + a_1 a_3 + a_2 a_3 = 1.
\]
Since $a_1^2 + a_2^2 + a_3^2=1$, this leads to
\begin{equation}\label{a1a2a3}
  (a_1 - a_2)^2 + (a_1 - a_3)^2 + (a_2 - a_3)^2 = 0
\end{equation}
which shows that $a_1 = a_2 = a_3 = 1/\sqrt{3}$ must hold.

Second, assume that two coordinates are equal, and the third is different from them. Without loss of generality, we may suppose that $a_2 = a_3$ (the same argument applies to the other cases as well). Then \eqref{a1a2_2} transforms to
\begin{equation}\label{a2a3eq}
  a_1(1 -  2 a_2^2 - a_1 a_2) = 0.
\end{equation}

Since $a_1^2 + 2 a_2^2 = 1$ and $a_1 >0$, this implies that $a_1 = a_2$, which contradicts our assumption.

Thus, all the critical directions are diagonal. As is well known, $(1,0,0)^\perp \cap Q_3$ is a minimal section, while $(\frac{1}{\sqrt{2}}, \frac{1}{\sqrt{2}},0)^\perp$ yields the maximum. By calculating the bordered Hessian one finds that 3-diagonal directions are saddle points (see also Figure~\ref{fig_3d}).
\end{proof}

\begin{figure}[h]
\centering
  \includegraphics[width=0.6\textwidth]{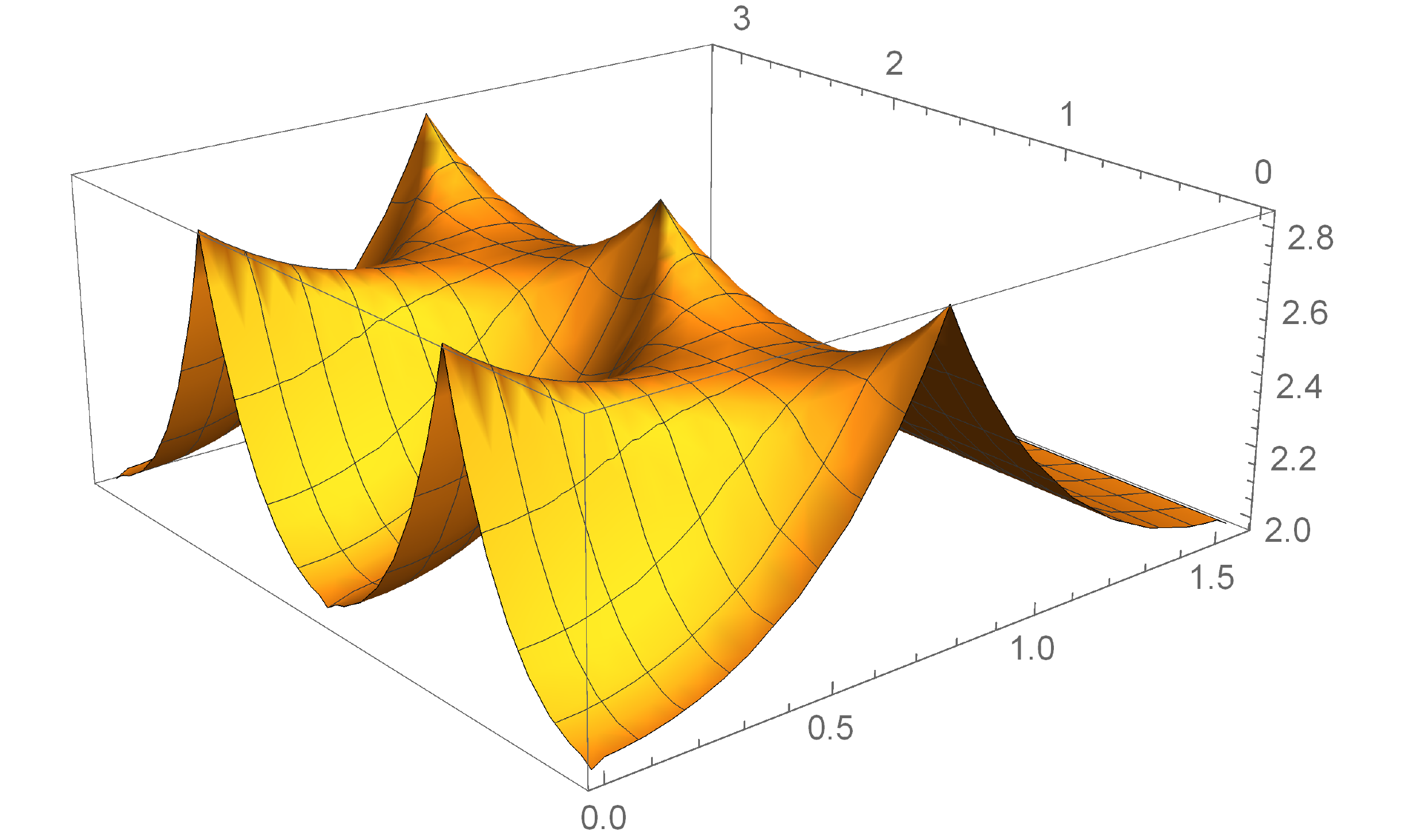}
  \caption{Plot of $\vol_2(Q_3 \cap (\sin \alpha, \cos \alpha \sin \beta, \cos \alpha \cos \beta)^\perp)$, the area of central sections of $Q_n$ for $n=3$,  with $\alpha \in [0, \pi/2]$ and $\beta \in [0, \pi]$.}\label{fig_3d}
\end{figure}

\section{Critical sections in the 4-dimensional case}

\begin{proof}[Proof of Theroem~\ref{thm3}]
Similarly to the 3-dimensional case, the argument is based on \eqref{a1a2}. Yet, the 4-dimensional case requires a longer discussion and case study. In order to simplify the subsequent arguments, we denote the coordinates with $b_1, b_2, b_3$ and $b_4$ momentarily,  which will be substituted by the coordinates of $\a$ in various order later.

Let thus $0<b_1 \leq b_2$ and $0 < b_3 \leq b_4$ with
\begin{equation}\label{sumbi}
b_1^2 + b_2^2 + b_3^2 + b_4^2= 1.
\end{equation} By comparing $b_1$ and $b_2$, \eqref{a1a2} shows that
\begin{equation}\label{b1b2}
  \frac{1 - b_2^2}{b_2} \int_{b_1 - b_2}^{b_1 + b_2} f_{b_3 X_3 + b_4 X_4}(x) \dd x = \frac{1 - b_1^2}{b_1} \int_{b_2 - b_1}^{b_2 + b_1} f_{b_3 X_3 + b_4 X_4}(x) \dd x.
\end{equation}
Note that
\begin{equation}\label{fx}
   f_{b_3 X_3 + b_4 X_4}(x) = \begin{cases}
    \frac{x + b_3 + b_4}{4 b_3 b_4} & \text{for }  - b_3 - b_4 \leq x \leq b_3 - b_4 \\
   \frac 1 {2 b_4} & \text{for } b_3 - b_4 \leq x \leq b_4 - b_3 \\
    \frac{-x + b_3 + b_4}{4 b_3 b_4} & \text{for } b_4 - b_3 \leq x \leq b_4 + b_3
  \end{cases}
\end{equation}

We will consider four cases according to the signs of $(b_1 + b_2) - (b_3 + b_4)$ and $(b_1 + b_4 ) - (b_2 + b_3)$.

\medskip
\noindent
{\bf Case A.} $b_1 + b_2 \leq b_3 + b_4$ and $b_2 + b_3 \leq b_1 + b_4$. Then, since $b_2  - b_1 \leq b_4 - b_3$, by \eqref{fx}, \eqref{b1b2} leads to
\[
(b_2 - b_1)[(b_1 + b_2 + b_3 - b_4)^2 (1+ b_1 b_2 ) - 8 b_1 b_2 b_3( b_1 + b_2) ].
\]
Therefore, either $b_1 = b_2$ holds true, or
\begin{equation}\label{caseA}
(b_1 + b_2 + b_3 - b_4)^2 (1+ b_1 b_2 ) = 8 b_1 b_2 b_3( b_1 + b_2) .
\end{equation}

\medskip
\noindent
{\bf Case B.} $b_1 + b_2 \leq b_3 + b_4$ and $b_2 + b_3 \geq b_1 + b_4$. Comparing $b_1$ and $b_2$ by \eqref{b1b2} now leads to
\begin{equation}\label{caseB}
(b_2^2 - b_1^2)(1 + b_2^2 - 2 b_2 (b_3 + b_4)) = (1 - b_2^2) ( b_3 - b_4)^2.
\end{equation}

\medskip
\noindent
{\bf Case C.} $b_1 + b_2 \geq b_3 + b_4$ and $b_2 + b_3 \leq b_1 + b_4$. In this case, \eqref{b1b2} simplifies to
\[
 b_1 (1 - b_2^2) \Big(\frac 1 2 + \frac {b_2 - b_1}{2 b_4} \Big) = b_2 (1 - b_1^2) \Big(\frac 1 2 - \frac {b_2 - b_1}{2 b_4} \Big).
\]
Thus, either $b_1 = b_2$, or
\begin{equation}\label{caseC}
  b_1 + b_2 - b_4- b_1 b_2^2 - b_1^2 b_2 - b_1 b_2 b_4 = 0.
\end{equation}

\medskip
\noindent
{\bf Case D.}  $b_1 + b_2 \geq b_3 + b_4$ and $b_2 + b_3 \geq b_1 + b_4$. Now, equation \eqref{b1b2} may be written as
\[
b_1 (1 - b_2^2) \Big(1 -  \frac {(b_1 - b_2 + b_3 + b_4)^2}{8 b_3 b_4} \Big) = b_2 (1 - b_1^2) \frac {(b_1 - b_2 + b_3 + b_4)^2}{8 b_3 b_4},
\]
which leads to
\begin{equation}\label{caseD}
  8 b_1 b_3 b_4 ( 1 - b_2^2) - (b_1 - b_2 + b_3 + b_4)^2 \big[b_1  + b_2 - b_1^2 b_2 - b_1 b_2^2  \big] =0.
\end{equation}

Let now $\a = (a_1, a_2, a_3, a_4) \in S^3$ be a critical direction with none of its coordinates being 0. We may assume that $0 <a_1 \leq a_2 \leq a_3 \leq a_4$.

We will  study two cases according the sign of $(a_1 + a_4) - ( a_2 + a_3)$.
Suppose first that $a_1 + a_4 \leq a_2 + a_3$. It is easy check that the conditions of Case C are satisfied in all the three substitutions below:
\begin{equation}\label{subs1}
(b_1, b_2 ,b_3, b_4):= (a_2, a_4, a_1, a_3) \textrm{ or } (a_3, a_4, a_1, a_2) \textrm{ or } (a_2, a_3, a_1, a_4).
\end{equation}
Applying these substitutions, we proceed  as in the proof of Theorem~\ref{thm2}.

If $a_2, a_3$ and $a_4$ are all different, then we have
\[
(a_2 - a_3)^2 + (a_2 - a_4)^2 + (a_3 - a_4)^2 = - 2 a_1^2
\]
which yields that $a_1 = 0$, a contradiction.

Assume next that two of $a_2, a_3$ and $a_4$ are equal, and the third one differs from them. The either $a_2 = a_3$ or $a_3 = a_4$. Suppose that $a_2 = a_3$, then $a_2 \neq a_4$. Applying the first substitution in $\eqref{subs1}$, equation \eqref{caseC} simplifies to
\[
a_4 ( 1 - 2 a_2^2 - a_2 a_4) =0,
\]
which by $a_4 >0$ and $1 - 2 a_2^2  = a_1^2 +  a_4^2$ implies that
\[
a_1^2 + a_4^2(a_4 - a_2) = 0,
\]
which is impossible.

Thus, we must have that $a_3 = a_4$ for all critical directions for which $a_1 + a_4 \leq a_2 + a_3$ holds. Apply the following substitution:
\[
(b_1, b_2 ,b_3, b_4):= (a_1, a_2,a_3, a_3)
\]
which belongs to Case B. Then \eqref{caseB} simplifies to
\[
(a_2^2 - a_1^2) (1 + a_2^2 - 4 a_2 a_3) =0.
\]
Therefore, either $a_1 = a_2$, or, using \eqref{sumbi},
\[
a_1^2 + 2 (a_2 - a_3)^2 = 0,
\]
which contradicts to $a_1 >0$.

Thus, $a_1 = a_2$ and $a_3 = a_4$. Taking the first substitution in \eqref{subs1} yields that either $a_1 = a_3$, in which case $\a$ is a 4-diagonal direction, or by \eqref{caseC},
\[
a_1(1 - a_1 a_3 - 2 a_3^2) =0.
\]
By \eqref{sumbi}, this is equivalent to
\[
a_1 ( 2 a_1 - a_3) =0
\]
which is only possible if $a_3 = 2 a_1$. Accordingly,
\begin{equation}\label{aeq}
  \a = \Big(\frac 1 {\sqrt{10}}, \frac 1 {\sqrt{10}}, \frac 2 {\sqrt{10}}, \frac 2 {\sqrt{10}}\Big).
\end{equation}
In this case, \eqref{a1a2} indeed holds for any pair of coordinates of $\a$, therefore, $\a$ is a critical  direction.

Second, assume that $a_1 + a_4 \geq a_2 + a_3$. In this case, conditions of Case D are satisfied for each of the following three substitutions:
\begin{equation}\label{subsb}
  (b_1, b_2 ,b_3, b_4):= (a_1, a_4,a_2, a_3)  \textrm{ or }  (a_2, a_4, a_1, a_3) \textrm{ or } (a_3, a_4, a_1, a_2).
\end{equation}
Taking the difference of \eqref{caseD} under the first two substitutions above, we obtain that
\[
(a_1 + a_2 + a_3 - a_4)^2 (a_1 - a_2) (a_4 (a_1 + a_2 + a_4) -1) = 0.
\]
Equation \eqref{absa} for $ k =4$ shows that the first term above may not be $0$. Thus, either $a_1 = a_2$, or
\begin{equation}\label{a1a2a4}
a_1 + a_2= \frac 1 {a_4} - a_4.
\end{equation}
Applying the same argument for the other two pairs of substitutions of \eqref{subsb} yields the same conclusion for $a_1$ and $a_3$, and for $a_2$ and $a_3$.

Assume first that $a_1, a_2$ and $a_3$ are all different. Then \eqref{a1a2a4} implies that $1/ a_4 - a_4 = a_1 + a_2 = a_1 + a_3$, which leads to $a_2 = a_3$, a contradiction.

Therefore, we deduce that out of the coefficients $a_1, a_2, a_3$, at least two must be equal. Let us first assume that not all three of these are the same. The subsequent argument is going to be symmetric with respect to permuting the coordinates 1,2,3, therefore we may assume that $a_1 = a_2$, and they differ from $a_3$.

By the analogue of \eqref{a1a2a4},  $(a_1 + a_3 + a_4) a_4= 1.$ By \eqref{sumbi}, we also have $2 a_1^2 + a_3^2 + a_4^2 =1$. Moreover, taking the substitution
\[
(b_1, b_2 ,b_3, b_4):= (a_1, a_3,a_1, a_4)
\]
which belongs to Case A, \eqref{caseA} implies that
\[
(2 a_1 + a_3 - a_4)^2 ( 1 + a_1 a_3) = 8 a_1^2 a_3 ( a_1 + a_3).
\]
Utilizing a computer algebra software reveals that the only positive solution to the system of polynomial equations
\begin{equation*}
    \begin{cases}
      (a_1 + a_3 + a_4) a_4= 1\\
      2 a_1^2 + a_3^2 + a_4^2 =1\\
      (2 a_1 + a_3 - a_4)^2 ( 1 + a_1 a_3) - 8 a_1^2 a_3 ( a_1 + a_3) =0
          \end{cases}
\end{equation*}
is given by $a_1 = \frac 1 {\sqrt{10}}$,  $a_3 = a_4 =  \frac 2 {\sqrt{10}}$, which yields \eqref{aeq} again.

Finally, assume that  $a_1 = a_2 = a_3$. Then by \eqref{sumbi} and \eqref{caseD} via \eqref{subsb}, $a_1$ and $a_4$ satisfies the following system of polynomial equations:
\begin{equation*}
    \begin{cases}
      3 a_1^2 + a_4^2 =1\\
      8 a_1^3 (1 - a_4^2) - (3 a_1 - a_4)^2 (a_1 + a_4) ( 1 - a_1 a_4) =0
          \end{cases}
\end{equation*}
Moreover, because of \eqref{absa} we also have $a_1 > \frac 1 {\sqrt{12}} \approx 0.2887$. Again utilizing a computer algebra software shows that among positive numbers, the above system of equations has two solutions: either $a_1 = a_4 = \frac 1 2$ which yields that $\a$ is a 4-diagonal direction, or $a_1 \approx 0.2142$ and $a_4 \approx 0.9286$ which does not satisfy the above constraint.

Thus, critical directions are either diagonal or they are of the form \eqref{aeq}, up to permutations and sign changes. It is easy to check that these indeed satisfy \eqref{a1a2} for any pair of coordinates. Out of these possibilities, 1-diagonal sections constitute global minima, 2-diagonal sections yield global maxima. As seen before, 3-diagonal directions are saddle points. Calculating the bordered Hessian at 4-diagonal directions shows that these constitue local, but not global, maxima.
\end{proof}

\section{Concluding remarks}

This piece of research stemmed from the recent result of Bartha, Fodor and Gonz\'alez Merino \cite{BFGM21} who proved that volumes of central $k$-diagonal sections of $Q_n$ form an increasing sequence for $k\geq 3$. It is natural to ask if that result can be used to determine minimal/maximal central sections of the cube. To that end, it would be sufficient to show that all directions $\a \in S^{n-1}$ which are critical with respect to $\sigma(\a)$ are diagonal. However, Theorem~\ref{thm3} shows that this is not true. Yet, in the 4-dimensional case, non-diagonal critical points constitute only saddle points. Therefore, the following question remains open: {\em Is it true that all  locally extremal central sections are diagonal?} If so, that would yield an alternate proof to the celebrated result of Ball~\cite{Ba86} via \cite{BFGM21}.

Even though it is not true that critical sections are all diagonal, it holds in the following `approximate asymptotic sense'. Assuming that the $a_i$'s are fairly equal, the Central Limit Theorem implies that $\sum_{i=3}^n a_i X_i$ is close to a standard normal variable. Let $g(x) = \frac 1 {\sqrt{2 \pi}} e^{ - x^2 /2}$ be the standard Gaussian density. Introduce the function
\[
G(r, s):= \frac{1 - r^2}{r} \int_{s - 2r}^s g(x) \dd x = \frac{1 - r^2}{2 r} ( \erf(s) - \erf(s - 2r)).
\]
Then \eqref{a1a2} reads approximately as
\[
G(a_1, a_1 + a_2) = G(a_2, a_1 + a_2).
\]
This can be shown to imply $a_1 = a_2$. Of course, this heuristic argument does not exclude non-diagonal critical directions when the distribution of $\sum_{i=3}^n a_i X_i$  differs substantially from the normal distribution.

I only became aware of the neat article of Ivanov and Tsiutsiurupa~\cite{IT21} after proving Theorem~\ref{thm1}. I find it reassuring that two entirely different approaches yield essentially the same result.

The methods used in the present paper may also be applied to other problems regarding volumes or perimeters of sections. In particular, the problem of estimating volumes of central sections of the simplex is subject to a forthcoming paper.

Finally, I would like to express my gratitude to F. Fodor, G. Ivanov, A. Koldobsky and H. König for the illuminating discussions on the topic, and to the anonymous referee for useful suggestions.

\bigskip

\noindent
{\sc Gergely Ambrus}
\smallskip

\noindent
{\em Alfréd Rényi Institute of Mathematics, Eötvös Loránd Research Network, Budapest, Hungary \\ and\\ Bolyai Institute, University of Szeged, Hungary}
\smallskip

\noindent
e-mail address: \texttt{ambrus@renyi.hu}


\begin{thebibliography}{99999999}
\bibitem[AHK19]{AHK19} A.  Akopyan, A. Hubard, and R. Karasev, {\em Lower and upper bounds for the waists of different spaces.} Topol. Methods Nonlinear Anal. {\bf 53} (2019), no. 2, 457--490.

\bibitem[A21]{Al21} I. Aliev, {\em On the volume of hyperplane sections of a $d$-cube}. Acta Math. Hungar. {\bf 163} (2021), no. 2., 547--551.

\bibitem[B86]{Ba86} K.M. Ball, {\em Cube slicing in $\R^n$}. Proc. Amer. Math. Soc. {\bf 97} (1986), 465--473.

\bibitem[B89]{Ba89} --------, {\em Volumes of section of cubes and related problems.} In: J. Lindenstrauss, V.D. Milman
(Eds.), Israel seminar on Geometric Aspects of Functional Analysis, Lectures Notes in Mathematics,
Vol. 1376, Springer, Berlin, 1989.

\bibitem[BFGM21]{BFGM21} F.Á. Bartha, F. Fodor, B. González Merino, {\em Central diagonal sections of the n-cube.} Int. Math. Res. Not. {\bf 2021} (2021), no. 4, 2861--2881.


\bibitem[BGMN05]{BGMN05} F. Barthe, O. Guédon, S. Mendelson, and A. Naor, {\em A probabilistic approach to the geometry of the $\ell_p^n$-ball.} Ann. Probab. {\bf 33} (2005), 480--513.

\bibitem[BK03]{BK03} F. Barthe and A. Koldobsky, {\em Extremal slabs in the cube and the Laplace transform.} Adv. Math. {\bf 174} (2003), 89--114.

\bibitem[E19]{E19} A. Eskenazis, {\em On extremal sections of subspaces of $L_p$. }  Discrete Comput. Geom.{\bf 65} (2021), 489--509.

\bibitem[IT21]{IT21} G. Ivanov and I. Tsiutsiurupa, {\em On the volume of sections of the cube.} Anal. Geom. Metr. Spaces {\bf 9} (2021), 1--18.

\bibitem[H71]{Ha71} H. Hadwiger, {\em Gitterperiodische Punktmengen und Isoperimetrie.} Monatsh. Math. {\bf 76} (1972), 410--418.

\bibitem[H79]{He79}D. Hensley, {\em Slicing the cube in $\R^n$ and probability.} Proc. Amer. Math. Soc. {\bf 73} (1979), no. 1.,  95--100.

\bibitem[K21]{K21} H. König, {\em Non-central sections of the simplex, the cross-polytope and the cube.} Adv. Math. {\bf 376} (2021), 107458.

\bibitem[KK11]{KK11} H. König and A. Koldobsky, {\em Volumes of low-dimensional slabs and sections of the cube.} Adv. Appl. Math. {\bf 47} (2011), 894--907.

\bibitem[KK12]{KK12} --------, {\em On the maximal measure of sections of the $n$-cube.} Geometric
Analysis, Mathematical Relativity, and Nonlinear Partial Differential Equations, Contemp. Math. {\bf 599} (2012), 123--155.

\bibitem[KK19]{KK19} --------, {\em On the maximal perimeter of sections of the cube.} Adv. Math. {\bf 346} (2019), 773--804.


\bibitem[K98]{K98} A. Koldobsky, {\em An application of the Fourier transform to sections of star bodies.} Israel J. Math. {\bf 106} (1998), 157--164.

\bibitem[K05]{Ko05} --------, {\em Fourier analysis in convex geomety.} Mathematical Surveys and Monographs {\bf 116}, AMS, Providence, RI, 2005.

\bibitem[KY08]{KY08} A. Koldobsky, V. Yaskin, {\em The interface between convex geometry and harmonic analysis.} CBMS Regional Conference Series in Mathematics {\bf 108}, AMS, Providence, RI, 2008.


\bibitem[L1812]{La1812} P.S. Laplace, {\em Théorie analytique des probabilités.} Paris, 1812.

\bibitem[LT20]{LT20} R. Liu, T. Tkocz, {\em A note on extremal noncentral sections of the cross-polytope.}  Adv. Appl. Math. {\bf 118} (2020), 102031.

\bibitem[MP88]{MP88} M. Meyer and A. Pajor, {\em Sections of the unit ball of $\ell_p^n$.} J. Funct. Anal. {\bf 80}
(1988), no. 1, 109--123.

\bibitem[MSZZ13]{MSZZ13} J. Moody, C. Stone, D. Zach, A. Zvavitch, {\em A Remark on Extremal Non-Central Sections of the Unit Cube.} Asympt. Geometr. Anal., Fields Inst. Comm. {\bf 68}, Springer, New York, 2013, pp.211--228.

\bibitem[NP00]{NP00} F.L. Nazarov and A.N. Podkorytov, {\em Ball, Haagerup, and distribution functions.} Complex analysis, operators, and
related topics, Springer, 2000, pp. 247–267.

\bibitem[P1913]{Po13} G. Pólya, {\em Berechnung eines bestimmten Integrals.} Math. Ann. {\bf 74} (1913), 204--212.


\bibitem[V79]{Va79} J. Vaaler, {\em A geometric inequality with applications to linear forms}, Pacific J. Math. {\bf 83} (1979), no. 2, 543–-553.

\bibitem[Z08]{Z08} A. Zvavitch, {\em  Gaussian measure of sections of dilates and shifts of convex bodies.} Adv. Appl. Math. {\bf 41} (2008), 247--254.

\end{thebibliography}
\end{document}